\documentclass{amsart}

\newtheorem{theorem}{Theorem}[section]

\newtheorem{proposition}[theorem]{Proposition}
\newtheorem{corollary}[theorem]{Corollary}

\theoremstyle{definition}

\theoremstyle{remark}
\newtheorem{remark}[theorem]{Remark}

\numberwithin{equation}{section}

\usepackage[pagebackref,colorlinks=true]{hyperref}  

\newcommand{\f}{\varphi}
\newcommand{\g}{\tilde{g}}
\newcommand{\n}{\nabla}

\newcommand{\ttt}{\tilde\tau}
\newcommand{\M}{(M,\A\f,\A\xi,\A\eta,\A{}g)}

\newcommand{\I}{\iota}

\newcommand{\R}{\mathbb R}

\newcommand{\F}{\mathcal{F}}
\newcommand{\LL}{\mathcal{L}}

\newcommand{\lm}{\lambda}

\newcommand{\al}{\alpha}
\newcommand{\bt}{\beta}
\newcommand{\DD}{\partial}

\newcommand{\D}{\mathrm{d}\hspace{-0.5pt}}

\DeclareMathOperator{\Div}{div} 
\DeclareMathOperator{\grad}{grad} 
\DeclareMathOperator{\Hess}{Hess} 
\DeclareMathOperator{\tr}{tr} 

\newcommand{\A}{\allowbreak{}}

\newcommand{\thmref}[1]{Theorem~\ref{#1}}

\newcommand{\corref}[1]{Corollary~\ref{#1}}
\newcommand{\propref}[1]{Proposition~\ref{#1}}

\begin{document}

\title[Ricci-like solitons with arbitrary and gradient potential ...]
{Ricci-like solitons with arbitrary potential and gradient almost Ricci-like solitons 
on Sasaki-like almost contact B-metric manifolds}


\author{Mancho Manev}
\address
{University of Plovdiv Paisii Hilendarski,
Faculty of Mathematics and Informatics, Department of Algebra and
Geometry, 24 Tzar Asen St., Plovdiv 4000, Bulgaria
\&
Medical University of Plovdiv, Faculty of Public Health,
Department of Medical Informatics, Biostatistics and E-Learning, 15A Vasil Aprilov Blvd.,
Plovdiv 4002, Bulgaria} 
\email{mmanev@uni-plovdiv.bg}




\begin{abstract}
Ricci-like solitons with arbitrary potential are introduced and studied
on Sasaki-like almost contact B-metric manifolds. 
It is proved that the Ricci tensor of such a soliton is the vertical component of both B-metrics multiplied by a constant.
It is established that gradient almost Ricci-like solitons have constant soliton coefficients. 
Explicit examples of Lie groups as manifolds of dimensions 3 and 5 equipped with the structures studied are 
provided.
\end{abstract}



%
%


\maketitle

\section*{Introduction}

Ricci soliton is a special self-similar solution of the Hamilton's Ricci flow and it is a natural generalization of the notion of Einstein metric. 
According to \cite{Ham88}, a (pse\-u\-do-)Riemannian manifold admits a \emph{Ricci soliton} if the metric, its Ricci tensor and the Lie derivative of the metric along a vector field (called potential) are linear dependent. 
If the coefficients of this dependence are functions then the soliton is called an \emph{almost Ricci soliton} 
\cite{PigRigRimSet11}. 
If a function exists so that the potential is its gradient, 
then the (almost) Ricci soliton is called a \emph{gradient (almost) Ricci soliton} 
(see, e.g.,	 \cite{ChoKno04}, \cite{DeMan19}, \cite{Shar}).

The topic became more popular after Perelman's proof of the Poincar\'e  conjecture, following Hamilton's program to use the Ricci flow (see \cite{Per03}).
Ricci solitons have been explored by a number of authors
(see, e.g., 
\cite{BagIng12}, \cite{BejCra14}, \cite{BroCalGarGav12}, \cite{CalFin15}, 
\cite{CalPer15JGP},
\cite{GalCra}, \cite{Ive93}, \cite{NagPre}, \cite{PokYadCha18}).

Ricci solitons are also of interest to physicists, and in physical literature are called \emph{quasi-Einstein} 
(see, e.g., \cite{ChaVal96}, \cite{Fri85}). 

The presence of the structure 1-form $\eta$ on manifolds with almost contact or almost paracontact structure motivates 
the need to introduce so-called $\eta$-Ricci solitons. 
Then, $\eta\otimes\eta$ is the restriction of the metric on the orthogonal complement 
to the (para)contact distribution, determined by the structure vector field $\xi$. 
By adding a term proportional to $\eta\otimes\eta$ into the defining equality of a Ricci soliton, 
it is defined the notion of \emph{$\eta$-Ricci soliton}, introduced in \cite{ChoKim}. 
Later, it has been studied on almost contact and almost paracontact manifolds by many authors 
(e.g., \cite{AyaYil}, \cite{Bla15}, \cite{Bla16}, \cite{BlaPerAceErd}, \cite{PraHad}).
For almost $\eta$-Ricci solitons, see, for example \cite{Bla18}, \cite{BlaPer}.

Our global goal is to study the differential geometry of almost contact B-metric manifolds investigated since 1993 \cite{GaMiGr}, \cite{ManGri93}.

Unlike almost contact metric manifolds, almost contact B-metric manifolds have two metrics 
that are mutually associated with structural endomorphism. The restrictions of both B-metrics 
on the orthogonal distribution to the contact distribution is $\eta\otimes\eta$. 
This is the reason for introducing in \cite{Man62} a further generalization of the notions of a Ricci soliton and an $\eta$-Ricci soliton, the so-called \emph{Ricci-like soliton}, using both B-metrics and $\eta\otimes\eta$.
There, we have explored these objects with potential Reeb vector field on some important kinds of manifolds under consideration: Einstein-like, Sasaki-like and having a torse-forming Reeb vector field.
In \cite{Man63}, we continue to study Ricci-like solitons, whose potential is the Reeb vector field 
 or pointwise collinear to it.

In the present paper, our goal is to investigate Ricci-like solitons  with arbitrary potential on
almost contact B-metric manifolds  of Sasaki-like type, 
as well as gradient almost Ricci-like solitons on these manifolds.

The paper is organized as follows.
In Section 1, we recall  basic definitions and properties of almost contact B-metric manifolds of Sasaki-like type
and obtain several immediate consequences. 
Section 2 includes some necessary results and a 5-dimensional example for a Ricci-like soliton with a potential Reeb vector field.  
In Section 3, we study 	Ricci-like solitons with an arbitrary potential. 
Then, we prove an identity for the soliton constants and a property of the potential, as well as that the Ricci tensor is a constant multiple of $\eta\otimes\eta$. 
For the 3-dimensional case, we find the values of the sectional curvatures of the special 2-planes with respect to the structure and construct an explicit example. 
In Section 4, we introduce gradient almost Ricci-like solitons on Sasaki-like manifolds and prove that their Ricci tensor has the same form as in the previous section. For the example in Section 3, we find a potential function to illustrate the obtained results.


\section{Sasaki-like almost contact B-metric manifolds}


A differentiable manifold $M$ of dimension $(2n+1)$, equipped with an almost contact structure $(\f,\xi,\eta)$ 
and a B-metric $g$ is called an \emph{almost contact B-metric manifold} and it is denoted by $\M$. 
More concretely, $\f$ is an endomorphism
of the tangent bundle $TM$, $\xi$ is a Reeb vector field, $\eta$ is its dual contact 1-form and
$g$ is a pseu\-do-Rie\-mannian
metric $g$  of signature $(n+1,n)$ satisfying the following conditions \cite{GaMiGr}
\begin{equation}\label{strM}
\begin{array}{c}
\f\xi = 0,\qquad \f^2 = -\I + \eta \otimes \xi,\qquad
\eta\circ\f=0,\qquad \eta(\xi)=1,\\[4pt]
g(\f x, \f y) = - g(x,y) + \eta(x)\eta(y),
\end{array}
\end{equation}
where $\I$ stands for the identity transformation on $\Gamma(TM)$.

In the latter equality and further, $x$, $y$, $z$, $w$ will stand for arbitrary elements of $\Gamma(TM)$ or vectors in the tangent space $T_pM$ of $M$ at an arbitrary
point $p$ in $M$.

The following equations are immediate consequences of \eqref{strM}  
\begin{equation}\label{strM2}
\begin{array}{ll}
g(\f x, y) = g(x,\f y),\qquad &g(x, \xi) = \eta(x),
\\[4pt]
g(\xi, \xi) = 1,\qquad &\eta(\n_x \xi) = 0,
\end{array}
\end{equation}
where $\n$ denotes the Levi-Civita connection of $g$.

The associated metric $\g$ of $g$ on $M$ is also a B-metric and it is defined by
\begin{equation}\label{gg}
\g(x,y)=g(x,\f y)+\eta(x)\eta(y).
\end{equation}

In  \cite{GaMiGr}, almost contact B-metric manifolds (also known as almost contact complex Riemannian manifolds) are classified with respect
to the (0,3)-tensor $F$ defined by
\begin{equation}\label{F=nfi}
F(x,y,z)=g\bigl( \left( \nabla_x \f \right)y,z\bigr).
\end{equation}
It has the following basic properties:
\begin{equation}\label{F-prop}
\begin{array}{l}
F(x,y,z)=F(x,z,y)
=F(x,\f y,\f z)+\eta(y)F(x,\xi,z)
+\eta(z)F(x,y,\xi),\\[4pt]
F(x,\f y, \xi)=(\n_x\eta)y=g(\n_x\xi,y).
\end{array}
\end{equation}
This classification consists of eleven basic classes $\F_i$, $i\in\{1,2,\dots,11\}$.  

%



In \cite{IvMaMa45}, it is introduced the type of a \emph{Sasaki-like} manifold among almost
contact B-metric manifolds. The definition condition is its complex cone to be a K\"ahler-Norden manifold, i.e. with a parallel complex structure.
A Sasaki-like manifold with almost
contact B-metric structure is determined by the condition
\begin{equation}\label{defSl}
\begin{array}{l}
\left(\nabla_x\f\right)y=-g(x,y)\xi-\eta(y)x+2\eta(x)\eta(y)\xi,
\end{array}
\end{equation}
which is equivalent to the following $\left(\nabla_x\f\right)y=g(\f x,\f y)\xi+\eta(y)\f^2 x$.

Obviously, Sasaki-like manifolds form a subclass of the class $\F_4$. 
Moreover,
the following identities are valid for it \cite{IvMaMa45}
\begin{equation}\label{curSl}
\begin{array}{ll}
\n_x \xi=-\f x, \qquad &\left(\n_x \eta \right)(y)=-g(x,\f y),\\[4pt]
R(x,y)\xi=\eta(y)x-\eta(x)y, \qquad &\rho(x,\xi)=2n\, \eta(x), \\[4pt]
R(\xi,y)z=g(y,z)\xi-\eta(z)y,\qquad 				&\rho(\xi,\xi)=2n,
\end{array}
\end{equation}
where $R$ and $\rho$ stand for the curvature tensor and the Ricci tensor of $\n$.

The corresponding curvature tensor of type $(0,4)$ is determined as usually by
$R(x,y,z,w)=g(R(x,y)z,w)$.

Further, we use an arbitrary basis $\left\{e_i\right\}$, $i\in\{1,2,\dots,2n+1\}$ of
$T_pM$, $p\in M$.

On an arbitrary almost contact B-metric manifold, there exists a $(0,2)$-tensor $\rho^*$, which is associated with
$\rho$ regarding $\f$. It is defined by
$\rho^*(y,z)=g^{ij}R(e_i,y,z,\f e_j)$ and due to the first equality in \eqref{strM2} $\rho^*$ is symmetric.

The following relation between $\rho^*$ and $\rho$ is valid for a Sasaki-like manifold
\begin{equation}\label{rho*-Sl}
\rho^*(y,z)=\rho(y,\f z)+(2n-1)g(y,\f z).
\end{equation}
It follows by taking the trace for $x=e_i$ and $w=e_j$ of the following property of a Sasaki-like manifold \cite{IvMaMa45}
\begin{equation}\label{curf-Sl}
\begin{array}{l}
R(x,y,\f z,w)-R(x,y,z,\f w)\\[4pt]
=\left\{g(y,z)-2\eta(y)\eta(z)\right\}g(x,\f w)
+\left\{g(y,w)-2\eta(y)\eta(w)\right\}g(x,\f z)\\[4pt]
-\left\{g(x,z)-2\eta(x)\eta(z)\right\}g(y,\f
w)-\left\{g(x,w)-2\eta(x)\eta(w)\right\}g(y,\f z).
\end{array}
\end{equation}

As a corollary of \eqref{rho*-Sl} we have that $\rho(y,\f z)=\rho(\f y,z)$, i.e. $Q\circ \f=\f\circ Q$,
where $Q$ is the Ricci operator, i.e. $\rho(y,z)=g(Qy,z)$.

The scalar curvature $\ttt$ of $\g$ is defined by $\ttt=\g^{ij}\g^{kl}\tilde{R}(e_i,e_k,e_l,e_j)$, where $\tilde{R}$ is the curvature tensor of $\g$ and $\g^{ij}=-\f^j_k g^{ik} + \xi^i\xi^j$ holds.
In addition, the associated quantity $\tau^*$ of $\tau$ with respect to $\f$ is determined by $\tau^*=g^{ij}\rho(e_i,\f e_j)$.
For them, using \eqref{curf-Sl} for a Sasaki-like manifold, we infer the following relation
\begin{equation}\label{tttt*-Sl}
\ttt=-\tau^*+2n.
\end{equation}

In \cite{Man3}, it is given the following relations between $\tau$ and $\ttt$
\begin{equation}\label{DtauDttt-Sl}
\D\tau\circ\f=\D\ttt+2(\tau-2n)\eta,\qquad 
\D\ttt\circ\f=-\D\tau+2(\ttt-2n)\eta.
\end{equation}
As corollaries we have
\begin{equation}\label{DtauDttt1-Sl}
\D\tau\circ\f^2=\D\ttt\circ\f,\qquad 
\D\ttt\circ\f^2=-\D\tau\circ\f,
\end{equation}
\begin{equation}\label{DtauDttt2-Sl}
\D\tau(\xi)=2(\ttt-2n),\qquad 
\D\ttt(\xi)=-2(\tau-2n).
\end{equation}

\begin{proposition}\label{prop:nQ-Sl}
On a Sasaki-like $\M$ of dimension $2n+1$, the following formulae for the Ricci operator $Q$ are valid
\begin{gather}\label{nQxiQ-aSl}
(\n_x Q)\xi=Q\f x-2n\,\f x,
\\[4pt]
\label{nxiQ=Q-aSl}
(\n_\xi Q)y=2Q\f y. 
\end{gather}
\end{proposition}
\begin{proof}
For a Sasaki-like manifold, according to \eqref{curSl}, the equalities $Q\xi=2n\,\xi$ and $\n_x \xi=-\f x$ holds. 
Using them, we obtain immediately the covariant derivative in \eqref{nQxiQ-aSl}.

Now, we apply $\n_z$ to the expression of $R(x,y)\xi$ in \eqref{curSl} 
and then, using the form of $\n\eta$ in \eqref{curSl}, we get the following
\[
\left(\n_z R\right)(x,y)\xi=R(x,y)\f z - g(y,\f z)x +g(x,\f z)y.
\]
We take the trace of the above equality for $z=e_i$ and $x=e_j$ and use \eqref{rho*-Sl} to obtain
\begin{equation}\label{tr12}
g^{ij} (\n_{e_i} R)(e_j,y)\xi=-Q\f y - 2n \f y.
\end{equation}
By virtue of the following consequence the second Bianchi identity
\[
g^{ij} (\n_{e_i} R)(\xi,y)e_j=\left(\n_y Q\right)\xi-\left(\n_\xi Q\right)y,
\]
the symmetries of $R$ and \eqref{nQxiQ-aSl},  we get \eqref{nxiQ=Q-aSl}.
\end{proof}

As consequences of \eqref{nQxiQ-aSl} and \eqref{nxiQ=Q-aSl} we obtain respectively
\begin{equation}\label{etaQxi-aSl}
\begin{array}{l}
\eta\bigl((\n_x Q)\xi\bigr)=0,\qquad
\eta\bigl((\n_\xi Q)y\bigr)=0.
\end{array}
\end{equation}

Let us recall  \cite{Man62}, an almost contact B-metric manifold $\M$ is said to be
\emph{Einstein-like} if its Ricci tensor $\rho$ satisfies
\begin{equation}\label{defEl}
\begin{array}{l}
\rho=a\,g +b\,\g +c\,\eta\otimes \eta
\end{array}
\end{equation}
for some triplet of constants $(a,b,c)$.
In particular, when $b=0$ and $b=c=0$, the manifold is called an \emph{$\eta$-Einstein manifold} and an \emph{Einstein manifold}, respectively.

If $a$, $b$, $c$ are functions on $M$, then the manifold is called \emph{almost Einstein-like}, \emph{almost $\eta$-Einstein} and \emph{almost Einstein}, respectively.

Tracing \eqref{defEl} and using \eqref{tttt*-Sl}, the scalar curvatures $\tau$ and $\ttt$ of an Einstein-like almost contact B-metric manifold have the form
\begin{equation}\label{tauEl}
\tau=(2n+1)a+b+c,\qquad \ttt=2n(b+1).
\end{equation}

For a Sasaki-like manifold  $\M$ with $\dim M = 2n+1$ and a scalar curvature $\tau$ regarding $g$, which is Einstein-like with a triplet of
constants $(a,b,c)$, the following equalities are given in \cite{Man62}:
\begin{equation}\label{abctau-ElSl}
a + b + c = 2n, \qquad \tau = 2n(a + 1). 
\end{equation}
Then, for $\tilde\tau$ on an Einstein-like Sasaki-like manifold we obtain
\begin{equation}\label{abctau*-ElSl}
\tilde\tau=2n(b+1)
\end{equation}
and because \eqref{tauEl}, \eqref{abctau-ElSl} and \eqref{abctau*-ElSl}, the expression \eqref{defEl} becomes
\begin{equation}\label{defElSl}
\begin{array}{l}
\rho=\left(\dfrac{\tau}{2n}-1\right)g +\left(\dfrac{\tilde\tau}{2n}-1\right)\g 
+\left(2(n+1)-\dfrac{\tau+\tilde\tau}{2n}\right)\eta\otimes \eta.
\end{array}
\end{equation}

\begin{proposition}\label{prop:El-Dtau}
Let $\M$ be a $(2n+1)$-dimensional Sasaki-like manifold. If it is almost Einstein-like  with functions $(a,b,c)$
then the scalar curvatures $\tau$ and $\tilde\tau$ of $g$ and $\g$, respectively, are constants 
\begin{equation}\label{El-Dtauxi}
\tau = const, \qquad \ttt=2n
\end{equation}
and $\M$ is $\eta$-Einstein with constants  
\[
(a,b,c)=\left(\frac{\tau}{2n}-1,\,0,\,2n+1-\frac{\tau}{2n}\right).
\]
\end{proposition}
\begin{proof}
If $\M$ is almost Einstein-like then $\rho$ has the form in \eqref{defEl}, where $(a,b,c)$ are a triad of functions. 
Then, according to \eqref{tttt*-Sl},  \eqref{defEl},  \eqref{tauEl} and 
the expression for $\rho(\xi,\xi)$ on a Sasaki-like manifold, given in \eqref{curSl}, 
we have the following
 \begin{equation}\label{ElSl-abc}
a+b+c=2n,\qquad \tau=2n(a+1),\qquad \ttt=2n(b+1).
\end{equation}

Using \eqref{curSl}, we can express $R(x,y)\xi$ and $R(x,\xi)y$ as follows
\begin{gather}
\begin{array}{l}\label{Rxyxi-Sl}
R(x,y)\xi =	\dfrac{1}{4n^2}\bigl\{2n\left[\eta(x) Qy - \eta(y) Qx + \rho(x,\xi)y - \rho(y,\xi)x \right]\\[6pt]
\phantom{R(x,y)\xi =}
+ (\tau-2n)\left[\eta(y)x -\eta(x)y\right]+ (\ttt-2n)\left[\eta(y)\f x -\eta(x)\f y\right]\bigr\},
\end{array}
\\[4pt]
\begin{array}{l}\label{Rxxiy-Sl}
R(x,\xi)y =	\dfrac{1}{4n^2}\bigl\{2n\left[\rho(x,y)\xi + g(x,y)Q\xi - \rho(y,\xi)x - \eta(y)Qx \right]\\[6pt]
\phantom{R(x,\xi)y=}
+ (\tau-2n)\left[\eta(y)x -g(x,y)\xi\right]\\[4pt]
\phantom{R(x,\xi)y=}
+ (\ttt-2n)\left[\eta(y)\f x -g(x,\f y)\xi\right]\bigr\}.
\end{array}
\end{gather}
Then, for $y=\xi$ in either of the last two equalities, we have
\begin{equation}\label{Rxxixi-Sl}
R(x,\xi)\xi =	\eta(x)\xi-\dfrac{1}{2n} Qx 
-\dfrac{1}{4n^2}\bigl\{[\tau-2n(2n+1)]\f^2 x - [\ttt-2n]\f x \bigr\}.
\end{equation}

After that, we compute the covariant derivative of $R(x,\xi)\xi$ with respect to $\n_z$. 
Since \eqref{defSl} and \eqref{curSl}, we obtain
\[
\begin{array}{l}
\left(\n_z R\right)(x,\xi)\xi =-\dfrac{1}{2n} \left\{(\n_z Q)x-\eta(x)Q\f z\right\}
-\dfrac{1}{4n^2}\bigl\{\D\tau(z)\f^2 x - \D\ttt(z)\f x \\[6pt]
\phantom{\left(\n_z R\right)(x,\xi)\xi =}
-[\tau-2n(2n+1)]g(x, \f z)\xi+[\ttt-2n]g(\f x, \f z)\xi\bigr\} - \eta(x)\f z,
	\end{array}
\]
which by taking the trace for $z=e_i$ and $x=e_j$ and \eqref{DtauDttt1-Sl} gives the following
\begin{equation}\label{trnRxixi-Sl}
\begin{array}{l}
g^{ij}g\bigl(\left(\n_{e_i} R\right)(e_j,\xi)\xi,y\bigr) = -\dfrac{1}{4n}\D\tau(y)
-\left\{\dfrac{\ttt}{2n}-1\right\}\eta(y).
\end{array}
\end{equation}

By virtue of the following consequence the second Bianchi identity
\begin{equation}\label{trnRyxixi-Sl}
\begin{array}{l}
g^{ij} g\bigl(\n_{e_i} R)(y,\xi)\xi,e_j\bigr) =
\eta\bigl(\left(\n_y Q\right)\xi\bigr)-\eta\bigl(\left(\n_\xi Q\right)y\bigr)
\end{array}
\end{equation}
and \eqref{etaQxi-aSl}, we have that the trace in the left side of \eqref{trnRyxixi-Sl} vanishes.
Then, \eqref{trnRxixi-Sl} and \eqref{trnRyxixi-Sl} imply
\begin{equation}\label{Dtauy-Sl}
\D\tau(y) = 
-2\{\ttt-2n\}\eta(y),
\end{equation}
which comparing with \eqref{DtauDttt2-Sl} implies 
\[
\D\tau(\xi)=0,\qquad \ttt=2n.
\]
The latter equalities together with \eqref{DtauDttt1-Sl} and \eqref{ElSl-abc} complete the proof.
\end{proof}


\section{Ricci-like solitons with potential Reeb vector field on Sasaki-like manifolds}

In \cite{Man62}, by a condition for Ricci tensor, it is introduced the notion of a Ricci-like soliton with potential $\xi$ on an almost contact B-metric manifold.

Now, we generalize this notion for a potential, which is an arbitrary vector field as follows.
We say that $\M$ admits a \emph{Ricci-like soliton with potential vector field $v$}
if the following condition is satisfied for a triplet of constants $(\lm,\mu,\nu)$
\begin{equation}\label{defRl-v}
\begin{array}{l}
\frac12 \mathcal{L}_{v} g  + \rho + \lm\, g  + \mu\, \g  + \nu\, \eta\otimes \eta =0,
\end{array}
\end{equation}
%
%
where $\mathcal{L}$ denotes the Lie derivative.

If $\mu=0$ (respectively, $\mu=\nu=0$), then \eqref{defRl-v} defines an \emph{$\eta$-Ricci soliton} 
(respectively, a \emph{Ricci soliton}) on $\M$. 
 
If $\lm$, $\mu$, $\nu$ are functions on $M$, then the soliton is called \emph{almost Ricci-like soliton}, \emph{almost $\eta$-Ricci soliton} and \emph{almost Ricci soliton}, respectively.

%


If $\M$ is Sasaki-like, we have
\[
\left(\mathcal{L}_{\xi} g\right)(x,y)=g(\n_x\xi,y)+g(x,\n_y\xi)=-2g(x,\f y),
\]
i.e.
$\frac12 \mathcal{L}_{\xi} g=-\g+\eta\otimes \eta$.  Then, because of \eqref{defRl-v}, $\rho$ takes the form
\begin{equation}\label{SlRl-rho}
	\rho = -\lm g + (1-\mu) \g  - (1+\nu) \eta\otimes \eta.
\end{equation}

\begin{theorem}[\cite{Man62}]\label{thm:RlSl}
Let $\M$ be a $(2n+1)$-dimensional Sasaki-like manifold and let $a$, $b$, $c$, $\lm$, $\mu$, $\nu$ be constants that satisfy the following equalities:
\begin{equation}\label{SlElRl-const}
a+\lm=0,\qquad b+\mu-1=0,\qquad c+\nu+1=0.
\end{equation}
Then, the manifold admits
a Ricci-like soliton with potential $\xi$ and constants $(\lm,\A\mu,\A\nu)$, where $\lm+\mu+\nu=-2n$,
if and only if
it is Einstein-like with constants $(a,b,c)$, where $a+b+c=2n$.

In particular, we get:
\begin{enumerate}
	\item[(i)]    The manifold admits an $\eta$-Ricci soliton with potential $\xi$ and constants $(\lm,0,-2n-\lm)$ if and only if
the manifold is Einstein-like with constants $(-\lm,1,\lm+2n-1)$.

	\item[(ii)]   The manifold admits a shrinking Ricci soliton with potential $\xi$ and constant $-2n$ if and only if
the manifold is Einstein-like with constants $(2n,1,-1)$.

	\item[(iii)]   The manifold is $\eta$-Einstein with constants $(a,0,2n-a)$  if and only if
it admits a Ricci-like soliton with potential $\xi$ and constants $(-a,1,a-2n-1)$.

	\item[(iv)]   The manifold is Einstein with constant $2n$ if and only if
it admits a Ricci-like soliton with potential $\xi$ and constants $(-2n,1,-1)$.
\end{enumerate}
\end{theorem}

\subsection{Example 1}

In Example 2 of \cite{IvMaMa45}, it is given
a Lie group $G$ of
dimension $5$ (i.e. $n=2$)
 with a basis of left-invariant vector fields $\{e_0,\dots, e_{4}\}$
and the corresponding Lie algebra is 
 defined as follows
\begin{equation*}\label{comEx1}
\begin{array}{ll}
[e_0,e_1] = p e_2 + e_3 + q e_4,\quad &[e_0,e_2] = - p e_1 -
q e_3 + e_4,\\[4pt]
[e_0,e_3] = - e_1  - q e_2 + p e_4,\quad &[e_0,e_4] = q e_1
- e_2 - p e_3,\qquad p,q\in\R.
\end{array}
\end{equation*}
After that $G$ is equipped with an almost contact B-metric structure defined by
\begin{equation*}\label{strEx1}
\begin{array}{l}
g(e_0,e_0)=g(e_1,e_1)=g(e_2,e_2)=-g(e_{3},e_{3})=-g(e_{4},e_{4})=1,
\\[4pt]
g(e_i,e_j)=0,\quad
i,j\in\{0,1,\dots,4\},\; i\neq j,
\\[4pt]
\xi=e_0, \quad \f  e_1=e_{3},\quad  \f e_2=e_{4},\quad \f  e_3=-e_{1},\quad \f  e_4=-e_{2}.
\end{array}
\end{equation*}
%
It is verified that the constructed almost contact B-metric manifold
$(G,\f,\allowbreak{}\xi,\allowbreak{}\eta,\allowbreak{}g)$ is Sasaki-like.

In \cite{Man62}, it is proved that 
%
$(G,\f,\allowbreak{}\xi,\allowbreak{}\eta,\allowbreak{}g)$
is $\eta$-Einstein with constants  
\begin{equation}\label{abcS}
(a,b,c)=(0,0,4).
\end{equation}
Moreover, it is clear that $\tau=\ttt=4$.

It is also found there that
 $(G,\f,\allowbreak{}\xi,\allowbreak{}\eta,\allowbreak{}g)$
admits a Ricci-like soliton with potential $\xi$ and constants
\begin{equation}\label{lmnS}
(\lm,\mu,\nu)=(0,1,-5).
\end{equation}
Therefore, 
this example
is in unison with \thmref{thm:RlSl} (iii) for $a=0$.


Moreover, the constructed 5-dimensional example of a Sasaki-like manifold 
with the results in \eqref{abcS} and \eqref{lmnS} 
supports also \thmref{thm:RlSl-v}, \propref{prop:Lvrho}, \thmref{thm:ElSlRl} 
and \corref{cor:012n} for the case of $v=\xi$ and $n=2$.

\section{Ricci-like solitons with arbitrary potential on Sasaki-like manifolds}

\begin{theorem}\label{thm:RlSl-v}
Let $\M$ be a $(2n+1)$-dimensional Sasaki-like manifold. If it admits a Ricci-like soliton with arbitrary potential vector field $v$ and constants $(\lm,\mu,\nu)$ then it is valid the following identities 
\begin{equation}\label{lmn-v}
\lm+\mu+\nu=-2n,
\end{equation}
\begin{equation}\label{nxiv}
\n_\xi v=-\f v.
\end{equation}
\end{theorem}

\begin{proof}
According to \eqref{defRl-v},  a Ricci-like soliton with arbitrary potential vector field $v$ is defined by
$
\bigl(\LL_v g\bigr)(y,z) = -2 \rho(y,z) - 2 \lm\,g(y,z)  - 2 \mu\,\g(y,z)  -2 \nu\,\eta(y)\eta(z).
$
Then, bearing in mind \eqref{curSl}, the covariant derivative with respect to $\n_x$ has the form
\begin{equation}\label{nLvg}
\begin{array}{l}
\bigl(\n_x \LL_v g\bigr)(y,z) = -\,2 \bigl(\n_x \rho\bigr)(y,z) 
-2\mu\{g(\f x,\f y)\eta(z)+g(\f x,\f z)\eta(y)\}\\[4pt]
\phantom{\bigl(\n_x \LL_v g\bigr)(y,z) =}
+2(\mu+\nu)\{g( x,\f y)\eta(z)+g( x,\f z)\eta(y)\}.
\end{array}
\end{equation}

We use of the following formula from \cite{Yano70} for a metric connection $\n$ 
\[
\bigl(\n_x \LL_v g\bigr)(y,z) =  
g\left((\LL_v \n)(x,y),z\right) + g\left((\LL_v \n)(x,z),y \right), 
\]
which due to symmetry of $\LL_v \n$ can read as
\begin{equation}\label{gLvn}
2g\bigl((\LL_v \n)(x,y),z\bigr) = \bigl(\n_x \LL_v g\bigr)(y,z) + 
\bigl(\n_y \LL_v g\bigr)(z,x) - \bigl(\n_z \LL_v g\bigr)(x,y).
\end{equation}
 
Applying \eqref{gLvn} to \eqref{nLvg}, we obtain
\begin{equation}\label{nLvg2}
\begin{array}{l}
g\bigl((\LL_v \n)(x,y),z\bigr)  = 
- \bigl(\n_x \rho\bigr)(y,z)- \bigl(\n_y \rho\bigr)(z,x) + \bigl(\n_z \rho\bigr)(x,y) \\[4pt]
\phantom{g\bigl((\LL_v \n)(x,y),z\bigr) =}
-2\mu\,g(\f x,\f y)\eta(z)+2(\mu+\nu)g( x,\f y)\eta(z).
\end{array}
\end{equation}
Setting $y=\xi$ in the equality above and using \eqref{nQxiQ-aSl}, \eqref{nxiQ=Q-aSl}, we get
\begin{equation}\label{Lvnxi}
(\LL_v \n)(x,\xi) = -2 Q\f x.
\end{equation}
The covariant derivative of the above equation by using of \eqref{curSl} has the form
\begin{equation}\label{nLvn}
\begin{array}{l}
\bigl(\n_y \LL_v \n\bigr)(x,\xi) = (\LL_v \n)(x,\f y) -2(\n_y Q)\f x +2\eta(x) Qy\\[4pt]
\phantom{\bigl(\n_y \LL_v \n\bigr)(x,\xi) =}
-4n\,g(x,y)-2(2n+1)\eta(x)\eta(y)\xi.
\end{array}
\end{equation}

We apply the latter equality to the following formula from \cite{Yano70}
\begin{equation}\label{LvR}
(\LL_v R)(x,y)z = \bigl(\n_x \LL_v \n\bigr)(y,z) - \bigl(\n_y \LL_v \n\bigr)(x,z) 
\end{equation}
and owing to symmetry of $\LL_v \n$, we obtain  
the following consequence of \eqref{Lvnxi}, \eqref{nLvn} and \eqref{LvR}
\begin{equation}\label{LvRxyxi}
\begin{array}{l}
g\bigl((\LL_v R)(x,y)\xi,z\bigr) =  -\,(\n_x \rho)(\f y,z)+(\n_{\f y} \rho)(x,z) -(\n_z \rho)(x,\f y)\\[4pt]
\phantom{g\bigl((\LL_v R)(x,y)\xi,z\bigr) =}
  +(\n_y \rho)(\f x,z)-(\n_{\f x} \rho)(y,z) +(\n_z \rho)(\f x, y)\\[4pt]
\phantom{g\bigl((\LL_v R)(x,y)\xi,z\bigr) =}
-2\,\eta(x)\rho(y,z)+2\,\eta(y)\rho(x,z).
\end{array}
\end{equation}
Plugging  $y=z=\xi$ in \eqref{LvRxyxi} and using \eqref{Lvnxi}, we obtain
\begin{equation}\label{LvRQ}
(\LL_v R)(x,\xi)\xi = 0.
\end{equation}

On the other hand, applying $\LL_v$ to the expression of $R(x,\xi)\xi$ from \eqref{curSl} 
and 
using \eqref{defRl-v}, as well as the formulae for $R(x,y)\xi$ and $R(\xi,y)z$ from the same referent equalities, we get
\begin{equation}\label{LvR-Sl}
(\LL_v R)(x,\xi)\xi = (\LL_v \eta)(x)\xi+g(x,\LL_v \xi)\xi-2 \eta(\LL_v \xi)x
\end{equation}
or in an equivalent form
\begin{equation}\label{LvR2-Sl}
(\LL_v R)(x,\xi)\xi = \{(\LL_v \eta)(x)+g(x,\LL_v \xi)-2 \eta(\LL_v \xi)\eta(x)\}\xi+2 \eta(\LL_v \xi)\f^2x.
\end{equation}

Comparing \eqref{LvRQ} and \eqref{LvR2-Sl}, we obtain the following system of equations
\[
(\LL_v \eta)(x)+g(x,\LL_v \xi)-2 \eta(\LL_v \xi)\eta(x)=0,\qquad \eta(\LL_v \xi)=0,
\]
i.e. 
\begin{equation}\label{LvR3-Sl}
(\LL_v \eta)(x)+g(x,\LL_v \xi)=0,\qquad \eta(\LL_v \xi)=0.
\end{equation}

According to \eqref{defRl-v} and $\rho(x,\xi)=2n\eta(x)$ from \eqref{curSl}, we have for a Sasaki-like manifold
\begin{equation}\label{Lvgxxi-SlRl}
(\LL_v g)(x,\xi)=-2 (\lm+\mu+\nu+2n)\eta(x)
\end{equation}
and as a consequence  for $x=\xi$ the following
\begin{equation}\label{Lvgxixi1-SlRl}
(\LL_v g)(\xi,\xi)=-2 (\lm+\mu+\nu+2n).
\end{equation}

The Lie derivative of $g(x,\xi)=\eta(x)$ with respect to $v$ gives 
\begin{equation}\label{Lvgxxi}
(\LL_v g)(x,\xi)=(\LL_v \eta)(x)-g\left(x,\LL_v \xi\right),
\end{equation}
which for $x=\xi$ leads to
\begin{equation}\label{Lvgxixi2-SlRl}
(\LL_v g)(\xi,\xi)=-2\eta\left(\LL_v \xi\right).
\end{equation} 

From \eqref{Lvgxixi1-SlRl} and \eqref{Lvgxixi2-SlRl} we obtain
\[
\eta(\LL_v \xi)=\lm+\mu+\nu+2n.
\]
The latter equality implies \eqref{lmn-v},
by virtue of the second equality in \eqref{LvR3-Sl}.

Substituting \eqref{lmn-v} in \eqref{Lvgxxi-SlRl} gives the vanishing of $(\LL_v g)(x,\xi)$ 
and because of \eqref{Lvgxxi} we have $(\LL_v \eta)(x)=g\left(x,\LL_v \xi\right)$. 
Hence, bearing in mind the first equality in \eqref{LvR3-Sl}, we get
\begin{equation}\label{Lvxi-SlRl}
\LL_v \xi = 0,
\end{equation} 
which together with $\n\xi=-\f$ from \eqref{curSl} completes the proof.
\end{proof}

\begin{proposition}\label{prop:Lvrho}
Let $\M$ be a $(2n+1)$-dimensional Sasaki-like manifold. If it admits a Ricci-like soliton with arbitrary potential  $v$ 
then the Ricci tensor $\rho$  of $g$ and the scalar curvatures $\tau$ and $\tilde\tau$ of $g$ and $\g$, respectively, satisfy the following equalities
\begin{equation}\label{Lvrhoxi=}
\left(\LL_v \rho\right)(x,\xi) = 0, \qquad \tau=2n, \qquad \ttt=const.
\end{equation}
\end{proposition}
\begin{proof}
By \eqref{LvR} we find the following
\begin{equation}\label{LvRxi}
\begin{array}{l}
g\bigl(\left(\LL_v R\right)(x,y)\xi,z\bigr) =	-\, g\bigl(\left(\LL_v \n\right)(x,\f y),z\bigr)
																														+ g\bigl(\left(\LL_v \n\right)(\f x, y),z\bigr)
\\[4pt]
\phantom{g\bigl(\left(\LL_v R\right)(x,y)\xi,z\bigr) =}
- 2 \left(\n_x \rho\right)(\f y, z) + 2 \left(\n_y \rho\right)(\f x, z)
\\[4pt]
\phantom{g\bigl(\left(\LL_v R\right)(x,y)\xi,z\bigr) =}
- 2 \eta(x)\rho(y,z) + 2 \eta(y)\rho(x,z).
\end{array}
\end{equation}
Taking the trace of the last equality for $x=e_i$ and $z=e_j$ and using \eqref{nLvg2}, \eqref{tttt*-Sl}, 
$\D\tau=2\Div\rho$, we obtain
successively 
\begin{equation*}
\begin{array}{l}
g^{ij}g\bigl(\left(\LL_v \n\right)(e_i,\f y),e_j\bigr) = - \D\tau(\f y),
\\[4pt]
g^{ij}g\bigl(\left(\LL_v \n\right)(\f e_i, y),e_j\bigr) = \D\ttt(y),\\[-6pt]
\end{array}
\end{equation*}
\begin{equation}\label{trLvR}
g^{ij}g\bigl(\left(\LL_v R\right)(e_i,y)\xi,e_j\bigr) =  \left(\LL_v \rho\right)(y,\xi)
\end{equation}
and therefore 
the following formula is valid 
\begin{equation}\label{LvrhoxiD}
\left(\LL_v \rho\right)(y,\xi) =	-\, \D\ttt(y)+2(\tau-2n)\eta(y),
\end{equation}
which for $y=\xi$ implies
\begin{equation}\label{Lvrhoxixi}
\left(\LL_v \rho\right)(\xi,\xi) =	-\, \D\ttt(\xi)+2(\tau-2n).
\end{equation}

On the other hand, according to \eqref{LvRQ} and \eqref{trLvR}, $\left(\LL_v \rho\right)(\xi,\xi)$ vanishes
and therefore \eqref{LvrhoxiD} and \eqref{Lvrhoxixi} imply 
\begin{equation}\label{Lvrhoxi}
\left(\LL_v \rho\right)(x,\xi) = \D\tilde\tau(\f^2x), \qquad \D\tilde\tau(\xi) = 2\tau- 4n.
\end{equation}
The latter equalities, due to \eqref{DtauDttt1-Sl} and \eqref{DtauDttt2-Sl}, imply consequently $\D\ttt(\xi)=0$ and
\[
\tau=2n, \qquad \ttt=const.
\]
In conclusion, because of \eqref{Lvrhoxi}, we infer the assertion.
\end{proof}

\begin{theorem}\label{thm:ElSlRl}
Let $\M$ be a $(2n+1)$-dimensional Einstein-like Sasaki-like manifold. 
If it admits a Ricci-like soliton with potential $v$ then the Ricci tensor is $\rho=2n\, \eta\otimes\eta$ and the scalar curvatures are $\tau=\ttt=2n$.
\end{theorem}
\begin{proof}
The assertion follows from \thmref{thm:RlSl-v}, \propref{prop:El-Dtau} and \propref{prop:Lvrho}. 
\end{proof}

\begin{corollary}\label{cor:012n}
Let $\M$, $\dim M = 2n+1$, be an Einstein-like Sasaki-like manifold. 
Then it is $\eta$-Einstein with constants $(0,0,2n)$, 
which is equivalent to the existence on $M$ of a Ricci-like soliton 
with potential $\xi$ and constants $(0,1,-2n-1)$.
\end{corollary}
\begin{proof}
Using \thmref{thm:ElSlRl}, we obtain the following expression
$
\LL_v g =-2\lm\,g-2\mu\,\g$ $+2(\lm+\mu)\eta\otimes\eta,
$
which holds for $\lm=0$, $\mu=1$ in the case $v=\xi$. 
Therefore \thmref{thm:RlSl} is restricted to its case (iii) and $a=0$. 
\end{proof}

Let us recall, every non-degenerate 2-plane (or section) $\bt$ with a basis $\{x,y\}$ with respect to $g$ in
$T_pM$, $p \in M$, has the following sectional curvature
\begin{equation}\label{sect}
k(\bt;p)=\frac{R(x,y,y,x)}{g(x,x)g(y,y)-[g(x,y)]^2}.
\end{equation}
A section $\bt$ is said to be \emph{$\f$-holomorphic} 
if the condition $\bt= \f \bt$
holds. 
Every $\f$-holomorphic section has a basis of the form $\{\f x,\f^2 x\}$. 
A section $\bt$ is called a \emph{$\xi$-section} if it has a basis of the form $\{x,\xi\}$.

\begin{theorem}\label{thm:dim3}
Let $\M$ be a $3$-dimensional Sasaki-like manifold. 
If it admits a Ricci-like soliton with potential $v$ then:
\begin{itemize}
	\item[(i)] the sectional curvatures of its $\f$-holomorphic sections are equal to $-1$; 
   \item[(ii)] the sectional curvatures of its $\xi$-sections are equal to $1$.
\end{itemize}
\end{theorem}
\begin{proof}
It is well known that the curvature tensor of a 3-dimensional manifold has the form
\begin{equation}\label{Rxyz-3dim}
\begin{array}{l}
R(x,y)z = g(y,z)Qx - g(x,z)Qy +\rho(y,z)x - \rho(x,z)y\\[4pt]
\phantom{R(x,y)z =}
-\dfrac{\tau}{2}\{g(y,z)x - g(x,z)y\}.
\end{array}
\end{equation}
Then, substituting $y=z=\xi$ and recalling \eqref{curSl}, we have
\begin{equation}\label{rho3dim}
\rho = \frac12\{(\tau-2)g-(\tau-6)\eta\otimes\eta\},
\end{equation}
which means that the manifold is $\eta$-Einstein.
Therefore, because of \thmref{thm:ElSlRl}, we have 
\[
\rho=2\, \eta\otimes\eta,\qquad \tau=\ttt=2.
\]

%
%
Substituting the latter two equalities for $\tau$ and $\rho$ in \eqref{Rxyz-3dim}, we get
\begin{equation}\label{Rxyz-3dim-Rl}
\begin{array}{l}
R(x,y)z = -\,[g(y,z)-2\,\eta(y)\eta(z)]x +2\,g(y,z)\eta(x)\xi \\[4pt]
\phantom{R(x,y)z =}
+[g(x,z) -2\,\eta(x)\eta(z)]y  - 2\,g(x,z)\eta(y)\xi.
\end{array}
\end{equation}

Using a basis $\{\f x,\f^2 x\}$ of an arbitrary  $\f$-holomorphic section, we calculate its sectional curvature by \eqref{sect}, replacing  $x$ and $y$ by $\f x$ and $\f^2 x$, respectively. Then, bearing in mind \eqref{strM} and \eqref{strM2}, we obtain 
$k(\f x,\f^2 x)=-1$. 

Similarly, for a $\xi$-section with a basis $\{x,\xi\}$, we get $k(x,\xi)=1$, which completes the proof.
%
\end{proof}

\begin{remark}
Examples of 3-dimensional Sasaki-like manifolds as a Lie group from type $Bia(VII_0)(1)$,
a matrix Lie group, 
an $S^1$-solvable extension on a K\"{a}hler-Norden 2-manifold, 
 and their geometrical properties are studied 
in \cite{HM-Annu}, \cite{HM-Facta}, \cite{HM-Acta}, \cite{HMMek}, respectively. 
\end{remark}

%

\subsection{Example 2}\label{Ex2}

Let us consider $M$ as a set of points in $\R^3$ with coordinates $(x^1,x^2,x^3)$ 
and let $M$ be equipped with an almost contact B-metric structure defined by
\begin{equation}\label{strEx1-loc}
\begin{array}{c}
g\left(\DD_1,\DD_1\right)=-g\left(\DD_2,\DD_2\right)=\cos 2x^3,\qquad 
g\left(\DD_1,\DD_2\right)=\sin 2x^3,
\\[4pt]
g(\DD_1,\DD_3)=g(\DD_2,\DD_3)=0,\qquad g(\DD_3,\DD_3)=1,
\\[4pt]
\f  \DD_1=\DD_2,\qquad \f \DD_2=-\DD_1,\qquad \xi=\DD_3,
\end{array}
\end{equation}
where $\DD_1$, $\DD_2$, $\DD_3$ denote briefly
$\frac{\DD}{\DD{x^1}}$,  
$\frac{\DD}{\DD{x^2}}$, $\frac{\DD}{\DD{x^3}}$, respectively.
Then, the vectors determined by
\begin{equation}\label{eiloci}
e_1=\cos x^3 \DD_1+\sin x^3 \DD_2,\qquad
e_2=-\sin x^3 \DD_1+\cos x^3 \DD_2,\qquad 
e_3=\DD_3
\end{equation}
form an orthonormal $\f$-basis of $T_pM$, $p\in M$, i.e.
\begin{equation}\label{strEx1}
\begin{array}{c}
g(e_1,e_1)=-g(e_2,e_2)=g(e_3,e_3)=1\\[4pt]
g(e_i,e_j)=0,\quad
i,j\in\{1,2,3\},\; i\neq j,
\\[4pt]
\f  e_1=e_2,\qquad \f e_2=-e_1,\qquad \xi=e_3.
\end{array}
\end{equation}

Immediately from \eqref{eiloci} we obtain the commutators of $e_i$ as follows 
\begin{equation}\label{com}
[e_0,e_1]=e_2,\qquad 
[e_0,e_2]=-e_1,\qquad [e_1,e_2]=0.
\end{equation}

Then, according to Example 1 in \cite{IvMaMa45} for $n=1$, the solvable Lie group 
of dimension $3$
 with a basis of left-invariant vector fields $\{e_1,e_2, e_3\}$
 defined by \eqref{com} and equipped with the $(\f,\xi,\eta,g)$-structure from \eqref{strEx1}
is a Sasaki-like almost contact B-metric manifold.

In the well-known way, we calculate the components of the Levi-Civita connection $\n$ for $g$
and from there the corresponding components 
$R_{ijkl}=R(e_i,e_j,e_k,e_l)$ and $\rho_{ij}=\rho(e_i,e_j)$
of the  curvature tensor $R$ and the Ricci tensor $\rho$, respectively. 
The non-zero ones of them are the following (keep in mind  the symmetries of $R$)
\begin{equation}\label{neij}
\n_{e_1} e_2=\n_{e_2} e_1= -e_3,\qquad
\n_{e_1} e_3= -e_2,\qquad
\n_{e_2} e_3= e_1;
\end{equation}
\begin{equation}\label{Rrho-Ex2}
R_{1221}=R_{1331}=-R_{2332}=1,\qquad \rho_{33}=2.
\end{equation}
The latter equality means that the Ricci tensor has the following form
\begin{equation}\label{rho=2etaeta}
\rho=2\eta\otimes\eta,
\end{equation}
i.e. the manifold is Einstein-like with constants $(a,b,c)=(0,0,2)$. 
Therefore, the scalar curvatures with respect to $g$ and $\g$ are
$\tau=\ttt=2$.

The values of $R_{ijkl}$ in \eqref{Rrho-Ex2} imply for the sectional curvatures
\[
k_{12}=-k_{13}=-k_{23}=-1,
\]
which supports \thmref{thm:dim3}.

Let us consider a vector field, determined by the following 
\begin{equation}\label{v}
\begin{array}{l}
v=v^1e_1 + v^2e_2 + v^3e_3, 
\\[4pt]
v^1=-\{c_1\cos x^3 +c_2\sin x^3\}x^1
+\{c_2\cos x^3 -c_1\sin x^3\}x^2 + \sin x^3,
\\[4pt]
v^2=-\{c_2\cos x^3 -c_1\sin x^3\}x^1
-\{c_1\cos x^3 +c_2\sin x^3\}x^2 + \cos x^3,
\\[4pt]
v^3=c_3,
\end{array}
\end{equation}
and $c_1$, $c_2$, $c_3$ are arbitrary constants.

Using \eqref{eiloci}, \eqref{strEx1}, \eqref{neij} and \eqref{v},  we obtain the following
\begin{equation}\label{nv}
\begin{array}{l}
\n_{e_1} v= -c_1e_1  -(c_2+c_3)e_2 - v^2 e_3,
\\[4pt]
\n_{e_2} v= (c_2+c_3)e_1  -c_1e_2 - v^1 e_3,
\\[4pt]
\n_{e_3} v= v^2e_1  -v^1 e_2,
\end{array}
\end{equation}
that allow us to calculate the components $\left(\LL_v g\right)_{ij}=\left(\LL_v g\right)(e_i,e_j)$ of the Lie derivative $\LL_v g$. Then, we get the following nonzero ones
\[
\begin{array}{l}
\left(\LL_v g\right)_{11}=
-\left(\LL_v g\right)_{22}=-2c_1,\qquad
\left(\LL_v g\right)_{12}=2(c_2+c_3),
\end{array}
\]
which implies that this tensor has the following expression
\begin{equation}\label{LvgEx2}
\LL_v g = -2\,c_1\,g   -2(c_2+c_3)\g   +2(c_1+c_2+c_3)\eta\otimes\eta. 
\end{equation}
Substituting the latter equality and \eqref{rho=2etaeta} in \eqref{defRl-v}, we
obtain that $(M,\f,\xi,\eta,g)$ admits a Ricci-like soliton with potential $v$ determined by \eqref{v}
and the potential constants are 
\[
\lm=c_1,\qquad 
\mu=c_2+c_3,\qquad
\nu=-c_1-c_2-c_3-2.
\]

These results 
are in accordance with \thmref{thm:RlSl-v} and  \thmref{thm:ElSlRl}. The conclusion in  
 \propref{prop:Lvrho} follows from \eqref{rho=2etaeta} 
and the subsequent formula 
$(\LL_v\rho)(x,\xi)=-2g(\f x,v)+2\eta(\n_x v)$, together with the equalities in \eqref{v} and \eqref{nv}.

\section{Gradient almost Ricci-like solitons}

Let us consider a Ricci-like soliton, defined by \eqref{defRl-v} with the condition $\lm$, $\mu$, $\nu$ to be functions on $M$.
If its potential $v$ is a gradient of a differentiable function $f$, i.e. $v=\grad f$, then the soliton is called a \emph{gradient almost Ricci-like soliton} of $\M$. In this case \eqref{defRl-v} is reduced to the following condition
\begin{equation}\label{Hess}
\Hess f+\rho+\lm g+\mu\g +\nu\eta\otimes\eta=0,
\end{equation}
where $\Hess$ denotes the Hessian operator with respect to $g$, i.e. $\Hess f$ is defined by
\begin{equation}\label{Hess2}
(\Hess f)(x,y) := (\n_x \D f)(y)=g(\n_x \grad f,y).
\end{equation}
Taking the trace of \eqref{Hess}, we obtain
 \begin{equation}\label{Lap}
\Delta f+\tau+(2n+1)\lm+\mu+\nu=0,
\end{equation}
where $\Delta:=\tr\circ\Hess$ is the Laplacian operator of $g$.
Also for the Laplacian of $f$, the formula $\Delta f = \Div(\grad f)$ is valid, where $\Div$ stands for the divergence operator.

The
gradient Ricci-like soliton is said to be \emph{trivial} when $f$ is constant. 
Further, we consider only non-trivial gradient Ricci-like solitons.

Equality \eqref{Hess} with the recall of \eqref{Hess2} provides the following
\begin{equation}\label{nv}
\n_x v=-Qx-\lm x-\mu \f x-(\mu+\nu)\eta(x)\xi,
\end{equation}
where $Q$ is the Ricci operator and $v=\grad f$.

\begin{theorem}\label{thm:grad}
Let $\M$ be a Sasaki-like almost contact B-metric  manifold of dimension $2n+1$.
If it admits a gradient almost Ricci-like soliton with functions $(\lm,\mu,\nu)$ and a potential function $f$,
then $\M$ has constant scalar curvatures $\tau=\ttt=2n$ for both B-metrics $g$ and $\g$, respectively, and its Ricci tensor is $\rho=2n\,\eta\otimes\eta$. 
\end{theorem}

\begin{proof}
Using \eqref{nv}, we compute the following curvature tensor
\begin{equation}\label{Rxyv-aSl}
\begin{array}{l}
R(x,y)v=-\left(\n_x Q\right)y+\left(\n_y Q\right)x\\[4pt]
\phantom{R(x,y)v=}
+\{\D\lm(y)+\mu\eta(y)\}x-\{\D\lm(x)+\mu\eta(x)\}y\\[4pt]
\phantom{R(x,y)v=}
+\{\D\mu(y)+(\mu+\nu)\eta(y)\}\f x-\{\D\mu(x)+(\mu+\nu)\eta(x)\}\f y\\[4pt]
\phantom{R(x,y)v=}
+\D(\mu+\nu)(y)\eta(x) \xi- \D(\mu+\nu)(x)\eta(y) \xi.
\end{array}
\end{equation}

The latter expression implies the following equality 
\begin{equation*}\label{R1-}
\begin{array}{l}
R(\xi,y)v=-\left(\n_\xi Q\right)y+\left(\n_y Q\right)\xi+\{\D\lm(\xi)+\mu\}\f^2 y-\{\D\mu(\xi)+\mu+\nu\}\f y\\[4pt]
\phantom{R(\xi,y)v=}
+\D(\lm+\mu+\nu)(y)\xi-\D(\lm+\mu+\nu)(\xi)\eta(y) \xi,
\end{array}
\end{equation*}
where we apply \eqref{nQxiQ-aSl} and \eqref{nxiQ=Q-aSl} and get 
\begin{equation}\label{R1}
\begin{array}{l}
R(\xi,y)v=-Q\f y +\{\D\lm(\xi)+\mu\}\f^2y-\{\D\mu(\xi)+\mu+\nu+2n\}\f y\\[4pt]
\phantom{R(\xi,y)v=3Q\f y}
+\D(\lm+\mu+\nu)(y)\xi-\D(\lm+\mu+\nu)(\xi)\eta(y) \xi.
\end{array}
\end{equation}

We put $z=v$ in the equality for $R(\xi,y)z$ in \eqref{curSl} and obtain the following expression 
\begin{equation}\label{Rxiyv-aSl}
R(\xi,y)v=\D f(y)\xi-\D f(\xi)y.
\end{equation}

Combining \eqref{R1} and \eqref{Rxiyv-aSl}, we find the following formula
\begin{equation}\label{nxiQ-aSl}
\begin{array}{l}
Q\f y=\{\D(\lm-f)(\xi)+\mu\}\f^2y-\{\D\mu(\xi)+\mu+\nu+2n\}\f y\\[4pt]
\phantom{Q\f y=}
+\D(\lm+\mu+\nu-f)(y)\xi-\D(\lm+\mu+\nu-f)(\xi)\eta(y) \xi.
\end{array}
\end{equation}

We apply $\eta$ of equality \eqref{nxiQ-aSl} and since $Q\circ\f=\f\circ Q$ for a Sasaki-like manifold, we obtain the following
\begin{equation}\label{lmnf+=-aSl}
\D(\lm+\mu+\nu-f)(y)=\D(\lm+\mu+\nu-f)(\xi)\eta(y),
\end{equation}
which changes \eqref{nxiQ-aSl} and \eqref{R1} as follows
\begin{equation}\label{Qfi-aSl}
Q\f y=\{\D(\lm-f)(\xi)+\mu\}\f^2y-\{\D\mu(\xi)+\mu+\nu+2n\}\f y,
\end{equation}
\begin{equation}\label{R1=}
\begin{array}{l}
R(\xi,y)v=-Q\f y +\{\D\lm(\xi)+\mu\}\f^2y-\{\D\mu(\xi)+\mu+\nu+2n\}\f y\\[4pt]
\phantom{R(\xi,y)v=-Q\f y}
+\{\D f(y)-\D f(\xi)\eta(y)\}\xi
\end{array}
\end{equation}
and therefore we have
\begin{equation}\label{Rxiyvz}
\begin{array}{l}
R(\xi,y,v,z)\,{=} - \rho(y,\f z) +\{\D\lm(\xi)+\mu\}g(\f y,\f z) 
\\[4pt]
\phantom{R(\xi,y,v,z)\,{=}}
- \{\D\mu(\xi)+\mu+\nu+2n\}g( y,\f z)\\[4pt]
\phantom{R(\xi,y,v,z)\,{=}}
+\{\D f(y)-\D f(\xi)\eta(y)\}\eta(z).
\end{array}
\end{equation}
%
%
%

On the other hand, the expression of $R(x,y)\xi$ from \eqref{curSl} and equality \eqref{Rxyv-aSl} imply respectively the following two equalities
\begin{gather}
R(x,y,\xi,v)=\D f(x)\eta(y)-\D f(y)\eta(x),\label{R3}
\\[4pt]
\begin{array}{l}\label{R4}
R(x,y,v,\xi)=-\eta\left((\n_x Q)y-(\n_y Q)x\right)+\D(\lm+\mu+\nu)(y)\eta(x)\\[4pt]
\phantom{R(x,y,v,\xi)=-\eta\left((\n_x Q)y-(\n_y Q)x\right)}
-\D(\lm+\mu+\nu)(x)\eta(y).
\end{array}
\end{gather}
By summation of the latter two equalities, we find the following formula
\begin{equation}\label{etaQ-aSl}
\begin{array}{l}
\eta\bigl((\n_x Q)y-(\n_y Q)x\bigr)=-\,\D(\lm+\mu+\nu-f)(x)\eta(y)\\[4pt]
\phantom{\eta\left((\n_x Q)y-(\n_y Q)x\right)=}
+\D(\lm+\mu+\nu-f)(y)\eta(x),
\end{array}
\end{equation}
which because of \eqref{lmnf+=-aSl} is simplified to the following form
\begin{equation}\label{etaQ=00-aSl}
\eta\bigl((\n_x Q)y-(\n_y Q)x\bigr)=0.
\end{equation}

%

On the other hand, the expression of $R(\xi,y)z$ from \eqref{curSl} 
yield
\[
R(\xi,y,z,v)\,{=}-\D f(\xi)g(\f y, \f z)-\{\D f(y)-\D f(\xi)\eta(y) \}\eta(z),
\]
which together with \eqref{Rxiyvz} and the form of $\rho(x,\xi)$ from \eqref{curSl} implies
\begin{equation}\label{rho-v}
\begin{array}{l}
\rho(y,z)=\{\D\mu(\xi)+\mu+\nu+2n\}g(\f y, \f z)+\{\D(\lm-f)(\xi)+\mu\}g(y,\f z)\\[4pt]
\phantom{\rho(y,z)=}
+2n\eta(y)\eta(z).
\end{array}
\end{equation}
The latter equality can be rewritten in the form
\[
\begin{array}{l}
\rho=-\,\{\D\mu(\xi)+\mu+\nu+2n\}g+\{\D(\lm-f)(\xi)+\mu\}\g\\[4pt]
\phantom{\rho=}
+\{4n+\nu-\D(\lm-\mu-f)(\xi)\}\eta\otimes\eta,
\end{array}
\]
which means that the manifold is almst Einstein-like with coefficient functions 
\begin{equation}\label{abc-aElSlRl}
\begin{array}{c}
a\,{=}-\D\mu(\xi)-\mu-\nu-2n,\qquad 
b=\D(\lm-f)(\xi)+\mu,\\[4pt]
c\,{=}-\D(\lm-\mu-f)(\xi)+\nu+4n.
\end{array}
\end{equation}
Then, using \eqref{tauEl}, we obtain
\begin{equation}\label{tauEl-Df}
\tau\,{=}-2n\{\D\mu(\xi)+\mu+\nu+2n-1\},\qquad \ttt=2n\{\D(\lm-f)(\xi)+\mu+1\}.
\end{equation}

Contracting \eqref{Rxyv-aSl} with respect to $x$, we obtain
\[
\rho(y,v)=\frac12 \D\tau(y)+2n\, \D\lm(y)+\D(\mu+\nu)(y)-\D\mu(\f y)-\{\D(\mu+\nu)(\xi)-2n\,\mu\}\eta(y)
\]
and consequently for $y=\xi$ we have
\begin{equation}\label{rhoxiv-Df}
\rho(\xi,v)=\frac12 \D\tau(\xi)+2n\, \D\lm(\xi)+2n\,\mu.
\end{equation}
We compute the left side of \eqref{rhoxiv-Df} by the formula $\rho(x,\xi)=2n\, \eta(x)$ from \eqref{curSl} 
and then from \eqref{rhoxiv-Df} and \eqref{DtauDttt2-Sl}, we obtain
\begin{equation}\label{ttt-Df}
\ttt=-2n \{\D(\lm-f)(\xi)+\mu-1\}.
\end{equation}
Comparing the latter equality with \eqref{tauEl-Df}, we have
\[
\D(\lm-f)(\xi)=-\mu,
\]
\[
\ttt=2n.
\]
The former equality implies $b=0$ in \eqref{abc-aElSlRl} and therefore the manifold is almost $\eta$-Einstein 
and the latter one 
means that $\D\ttt=0$ and using \eqref{DtauDttt2-Sl}, we obtain for $\tau$ the following
\[
\tau=2n.
\] 
Then, substituting the value of $\tau$ in \eqref{tauEl-Df}, we obtain
\[
\D\mu(\xi)=-\mu-\nu-2n,
\]
which implies $a=0$ in \eqref{abc-aElSlRl} and finally we get $(a,b,c)=(0,0,2n)$.
\end{proof}



\subsection{Example 3}

Let $\M$ be the 3-dimensional Sasaki-like manifold, given in Example~2 of \S\ref{Ex2}.
Now, let $f$ be a differentiable function on $M$, defined by
\begin{equation}\label{f}
f=-\frac{1}{2}s\,\{(x^1)^2+(x^2)^2\} +x^2 + t\,x^3
\end{equation}
for arbitrary constants $s$ and $t$.
Then, the gradient of $f$ with respect to the B-metric $g$ is the following
\begin{equation}\label{gradf}
\begin{array}{l}
\grad f = -\,\{s\,x^1\cos x^3 + (s\,x^2-1)\sin x^3\}e_1\\[4pt]
\phantom{\grad f =}
+\{s\,x^1\sin x^3 - (s\,x^2-1)\cos x^3\}e_2 + t\, e_3.
\end{array}
\end{equation}

Using \eqref{eiloci}, we compute 
the  components of 
$\LL_{\grad f}g$ as follows
\[
\left(\LL_{\grad f}g\right)_{11}=-\left(\LL_{\grad f}g\right)_{22}=-2s,\qquad
\left(\LL_{\grad f}g\right)_{12}=2t,
\] 
which give us the following expression 
\[
\left(\LL_{\grad f}g\right)=-2s\, g -2t \g +2(s+t)\,\eta\otimes\eta.
\]

The latter equality coincides with \eqref{LvgEx2} for $s=c_1$, $t=c_2+c_3$. 
Therefore, 
$(M,\f,\xi,\eta,g)$ admits a Ricci-like soliton with potential $v=\grad f$ determined by \eqref{gradf}
and the potential constants are 
\[
\lm=s,\qquad 
\mu=t,\qquad
\nu=-s-t-2.
\]

In conclusion, the constructed 3-dimensional example of a Sasaki-like manifold 
with $\tau=\ttt=2$ and gradient Ricci-like soliton   
supports also \thmref{thm:grad}.

\subsection*{Acknowledgment}
The author was supported by projects MU19-FMI-020 and FP19-FMI-002 of the Scientific Research Fund,
University of Plovdiv Paisii Hilendarski, Bulgaria.

\end{document}